\documentclass[a4paper]{article}

\usepackage[small,bf,hang]{caption} 

\usepackage{amsmath,amsthm,amssymb}
\usepackage{graphicx,subcaption, tikz,color}
\usepackage{comment,xspace,paralist}							                              
\usepackage[shortlabels]{enumitem}	
\usepackage{hyperref}       
\hypersetup{
    colorlinks=true,
    citecolor = blue,
    linkcolor=blue,
    filecolor=magenta,
    urlcolor=cyan,
    bookmarks=true,
}
\usepackage{times}
\usepackage{todonotes}
\usepackage[ruled,linesnumbered]{algorithm2e}
\newtheorem{theorem}{Theorem}

\newtheorem{proposition}{Proposition}

\title{{\bf  Short proofs on $k$-extendible graphs}}

\author{
Shenwei Huang\thanks{College of Computer Science, Nankai University, Tianjin 300350, China.  Partially supported by the National Natural Science Foundation of China (11801284) and Natural Science Foundation of Tianjin (20JCYBJC01190).}
\and Yongtang Shi\thanks{Center for Combinatorics and LPMC, Nankai University, Tianjin 300071, China. Partially supported by the National Natural Science Foundation of China (No. 11922112) and Natural Science Foundation of Tianjin (Nos. 20JCZDJC00840, 20JCJQJC00090).}
}

\date{October 7, 2021}

\begin{document}

\maketitle

\begin{abstract} 
In this note, we give short inductive proofs
of two known results on $k$-extendible graphs
based on a property proved in [Qinglin Yu, A note on $n$-extendable graphs. Journal of Graph Theory, 16:349-353, 1992].
\end{abstract}



\section{Introduction}

A graph $G$ is {\em $k$-extendible} if it satisfies the following conditions:

\begin{itemize}
\item $|G|\ge 2k+2$;
\item $G$ is connected;
\item $G$ has a perfect matching;
\item for every matching $M_k$ of $G$ of size $k$, there is a perfect matching of $G$ containing $M_k$.
\end{itemize}

The notion of $k$-extendible graphs was first defined and studied by Plummer \cite{Pl80}.
In particular, 2-extendible bipartite graphs play an important role in the study of P\'olya's permanent problem \cite{Po1913} 
whose solution was obtained by Robertsen, Seymour and Tomas \cite{RST99} and independently by McCuaig \cite{Mc04}.
We refer to the monograph of Lov\'asz and Plummer \cite{LP86} for a detailed account of 1-extendible graphs.

\noindent {\bf Our Contribution.} 
Based on a property of $k$-extendible graphs proved by Yu \cite{Yu92}, we give short inductive proofs of two known results on $k$-extendible graphs.  
Our proofs are much simpler than the existing proofs due to the fact that
the property allows us to apply the inductive hypothesis on subgraphs of the given $k$-extendible graph
rather than on the given graph itself.


\section{Preliminaries}\label{sec:pre}

Let $G=(V,E)$ be a graph. For $S\subseteq V$ and a subgraph $H$ of $G$, {\em the neighborhood of $S$ in $H$}, denoted by $N_H(S)$,
is the set of vertices in $H$ that are adjacent to some vertex in $S$. The size of $N_H(S)$ is denoted by $d_H(S)$.
If $S=\{v\}$, we simply write $N_H(v)$ and $d_H(v)$ instead of $N_H(\{v\})$ and $d_H(\{v\})$, respectively.
For a subset $S\subseteq V$, we denote by $G[S]$ the subgraph induced by $S$.
Given a matching $M$ of $G$, we denote by $V(M)$ the set of vertices that are endvertices of edges in $M$.
The minimum degree and the matching number of $G$ are denoted by $\delta(G)$ and $\alpha'(G)$, respectively.
For other standard terminology we refer to \cite{BM08}.

We start with two simple propositions of $k$-extendible graphs that were obtained by Plummer \cite{Pl80}.
Since the proofs are short,  we include them here for the sake of completeness.

\begin{proposition}[\cite{Pl80}]\label{prop:k-extendible implies (k-1)-extendible}
Every $k$-extendible graph is $(k-1)$-extendible.
\end{proposition}

\begin{proof}
Let $G$ be a $k$-extendible graph.
By contradiction, let $M=\{a_1b_1,\ldots,a_{k-1}b_{k-1}\}$ be a matching of size $k-1$ that is not contained in a perfect matching of $G$. Since $G$ is $k$-extendible, $M$ is a maximal matching of $G$. This implies that 
\[
	S=V(G)\setminus \{a_1,b_1,\ldots,a_{k-1},b_{k-1}\}
\] 
is independent. Since $|G|\ge 2k+2$,  it follows that $|S|\ge 4$.
Since $M$ is not a maximum matching of $G$, it follows from {Berge's Theorem} (see \cite{BM08}) that
there exists an {\em $M$-augmenting path}, that is, an $u$-$v$ path $P$ such that $u,v\notin V(M)$
and edges of $P$ are alternating between $E(G)\setminus M$ and $M$, starting with an edge not in $M$. 
Then $M'=(E(P)\setminus M)\cup (M\setminus E(P))$
is a matching of size $k$ with $V(M')=V(M)\cup \{u,v\}$.
Since $G-V(M')=S\setminus \{u,v\}$ is an independent set of size at least 2, 
$M'$ cannot be extended to a perfect matching of $G$. This is a contradiction.
\end{proof}

\begin{proposition}[\cite{Pl80}]\label{prop:1-extendible graphs are 2-connected}
Every 1-extendible graph is 2-connected.
\end{proposition}

\begin{proof}
Suppose by contradiction that $G$ is a 1-extendible graph but not 2-connected.
Then there exists a cut vertex $v$ such that $G-v$ has components $C_1,\ldots,C_t$ for some $t\ge 2$.
Since $G$ is connected, $v$ has a neighbor $u_i\in C_i$ for each $1\le i\le t$.
Since $G$ is 1-extendible, there is a perfect matching containing $vu_1$ and this implies that $|C_1|$ is odd.
On the other hand, there is a perfect matching containing $vu_2$ and this implies that $|C_1|$ is even.
This is a contradiction.
\end{proof}

The following property of  $k$-extendible graphs was proved by Yu \cite{Yu92} whose proof used \autoref{lem:k-extendible graphs are (k+1)-connected} below. Here we give a new proof that avoids the use of \autoref{lem:k-extendible graphs are (k+1)-connected}.
\begin{proposition}[\cite{Yu92}]\label{prop:subgraphs of $k$-extendible graphs}
Let $G$ be a $k$-extendible graph with $k\ge 2$. Then for every edge $uv\in E(G)$, $G-\{u,v\}$ is $(k-1)$-extendible.
\end{proposition}

\begin{proof}
Let $e=uv\in E(G)$ and $G'=G-\{u,v\}$. 
By \autoref{prop:k-extendible implies (k-1)-extendible}, $G$ is 1-extendible and so 2-connected by \autoref{prop:1-extendible graphs are 2-connected}.
Since $G$ is $k$-extendible, every matching of size $k-1$ of $G'$ can be extended to a perfect matching of $G'$.
So it remains to show that $G'$ is connected.
Suppose by contradiction that $G'$ has components $C_1,C_2,\ldots,C_t$ for $t\ge 2$.
Since $G$ is 2-connected, each of $u$ and $v$ has a neighbor in each component $C_i$.
Let $s\in C_1$ be a neighbor of $u$ and $t\in C_2$ be a neighbor of $v$. 
Since $G$ is $k$-extendible with $k\ge 2$, there is a perfect matching of $G$ containing $\{us,vt\}$ by \autoref{prop:k-extendible implies (k-1)-extendible}.
This implies that $|C_1|$ is odd. Then there is no perfect matching of $G$ containing $uv$.
This contradicts that $G$ is 1-extendible. Therefore, $G-\{u,v\}$ is $(k-1)$-extendible.
\end{proof}

\section{New Proofs}

In this section, we present our new proofs of two known results on $k$-extendible graphs.
The first result was proved by Plummer \cite{Pl80} on the connectivity of $k$-extendible graphs.
The overall strategy of Plummer \cite{Pl80} 
was to apply the inductive hypothesis on the input graph (due to \autoref{prop:k-extendible implies (k-1)-extendible}) and then used a variation of Menger's Theorem. 
Our proof below, on the other hand, is simpler due to the fact that 
we were able to apply the inductive hypothesis on subgraphs of the input graph
due to \autoref{prop:subgraphs of $k$-extendible graphs}.

\begin{theorem}[\cite{Pl80}]\label{lem:k-extendible graphs are (k+1)-connected}
Every $k$-extendible graph is $(k+1)$-connected.
\end{theorem}

\begin{proof}[Our Proof]
Let $G$ be a $k$-extendible graph.
We prove by induction on $k$.
The base case is \autoref{prop:1-extendible graphs are 2-connected}.
Now suppose that $k\ge 2$ and the statement is true for $(k-1)$-extendible graphs. 
By \autoref{prop:subgraphs of $k$-extendible graphs}, $G-\{u,v\}$ is $(k-1)$-extendible
for every edge $uv\in E(G)$ and so is $k$-connected by the inductive hypothesis.
By \autoref{prop:k-extendible implies (k-1)-extendible} and \autoref{prop:1-extendible graphs are 2-connected},
it follows that $\delta(G)\ge 2$. For any vertex $v\in V(G)$, let $u$ be a neighbor of $v$.
Since $d(u)\ge 2$, $u$ has a neighbor $w$ other than $v$. 
Since $H=G-\{u,w\}$ is $k$-connected, $d_H(v)\ge k$ and thus $d_G(v)\ge k+1$.
This shows that $\delta(G)\ge k+1$.

Now let $S\subseteq V(G)$ be an arbitrary set with $|S|=k$. Let $s\in S$ and $t$ be a neighbor of $s$.
We show that $G-S$ is connected.

\noindent {\bf Case 1.} $t\in S$. Then 
\[
	G-S=(G-\{s,t\})-(S\setminus \{s,t\})
\]
is connected, since $G-\{s,t\}$ is $k$-connected.

\noindent {\bf Case 2.} $t\notin S$. Let $G'=G-(S\cup \{t\})$. Note that 
\[
	G'=(G-\{s,t\})-(S\setminus \{s\}).
\]
Since $G-\{s,t\}$ is $k$-connected, $G'$ is connected. Since $\delta(G)\ge k+1$, $t$ has a neighbor in $G'$. Therefore, 
\[
	G-S=G[V(G')\cup \{t\}]
\]
is connected. 
\end{proof}

The second result is on $k$-extendible bipartite graphs.
The celebrated Hall's Theorem gives a necessary and sufficient condition for a balanced bipartite graph to have a perfect matching.
It turns out that $k$-extendible bipartite graphs have a similar characterization.

\begin{theorem}[\cite{BP71}]\label{thm:characterization of $k$-extendible bipartite graphs}
Let $G=(X,Y)$ be a connected bipartite graph with a perfect matching and $|G|\ge 2k+2$. 
Then $G$ is $k$-extendible if and only if $|N(A)|\ge |A|+k$ for every subset $A\subseteq X$ with $1\le |A|\le |X|-k$.
\end{theorem}

\autoref{thm:characterization of $k$-extendible bipartite graphs} was first stated and proved by Brualdi and Perfect \cite{BP71}
in the language of matrices (Theorem 2.1 in \cite{BP71}). Here we give two graph-theoretical proofs. 
The first one relies on \autoref{prop:subgraphs of $k$-extendible graphs}
while the second one is based on the K\"{o}nig-Ore Formula.

\begin{theorem}[The K\"{o}nig-Ore Formula]
Let $G=(X,Y)$ be a bipartite graph. Then
\[
	\alpha'(G)=|X|-\max_{S\subseteq X}(|S|-|N(S)|).
\]
\end{theorem}

\begin{proof}[Our First Proof of \autoref{thm:characterization of $k$-extendible bipartite graphs}]
We first prove the sufficiency. 
Take a matching 
\[
	M=\{x_1y_1,\ldots,x_ky_k\}
\] 
of size $k$. Let $X'=X\setminus \{x_1,\ldots,x_k\}$
and $Y'=Y\setminus \{y_1,\ldots,y_k\}$. Denote by $H$ the subgraph of $G$ induced by $X'\cup Y'$. Note that every nonempty subset
$A$ of $X'$ has $1\le |A|\le |X|-k$. It follows from the assumption that $|N_G(A)|\ge |A|+k$. This implies that $|N_H(A)|\ge |A|$.
By the K\"{o}nig-Ore Formula, $H$ has a perfect matching $M_H$. 
It follows that $M_H\cup M$ is a perfect matching of $G$ containing $M$.
This shows that $G$ is $k$-extendible.

We now prove the necessity by induction on $k$.

\noindent {\bf Base Case:} $k=1$. By contradiction, let $A$ be a subset of $X$ with $1\le |A|\le |X|-1$ such that $|N(A)|<|A|+1$.
Since $G$ has a perfect matching, $|N(A)|\ge |A|$. It then follows that $|N(A)|=|A|$. Since $G$ is connected, there is an edge $e=xy$
between $N(A)$ and $X\setminus A$. So there is no perfect matching of $G$ containing $e$, simply because there are not enough
vertices in $G-\{x,y\}$ to match vertices in $A$. 

\noindent {\bf Inductive Step:} 
We assume that $k\ge 2$ and the statement is true for $k-1$. 
let $A$ be an arbitrary subset of $X$ with $1\le |A|\le |X|-k$. If $N(A)=Y$, then
\[
	|N(A)|=|Y|=|X|\ge |A|+k.
\]
So we may assume that $N(A)\neq Y$. Since $G$ is connected, there is an edge $xy\in E(G)$ such that
$y\in N(A)$ and $x\in X\setminus A$. Let $G'=G-\{x,y\}=(X\setminus \{x\},Y\setminus \{y\})$.
By \autoref{prop:subgraphs of $k$-extendible graphs},
$G'$ is $(k-1)$-extendible. On the other hand, $A\subseteq X\setminus \{x\}$ has
\[
	1\le |A|\le |X|-k=(|X|-1)-(k-1).
\]
By the inductive hypothesis, $|N_{G'}(A)|\ge |A|+(k-1)$. Since $N_G(A)=N_{G'}(A)\cup \{y\}$, it follows that
$|N_{G}(A)|\ge |A|+k$.
\end{proof} 

\begin{proof}[Our Second Proof of \autoref{thm:characterization of $k$-extendible bipartite graphs}]

The difference lies in the inductive step of the necessity.
We assume that $k\ge 2$ and the statement is true for $k-1$. 
By contradiction,  
let $A$ be a subset of $X$ with $1\le |A|\le |X|-k$ such that $|N(A)|<|A|+k$. 
By \autoref{prop:k-extendible implies (k-1)-extendible}, $G$ is $(k-1)$-extendible. 
By the inductive hypothesis,
$|N(A)|\ge |A|+(k-1)$. It follows that
\begin{equation}\label{equ:1}
|N(A)|=|A|+(k-1).
\end{equation}
Let $X'=X\setminus A$ and write $B=N(A)$. Note that $|B|\ge k$ for otherwise \autoref{equ:1} would be contradicted. 
Denote by $H$ the subgraph
induced by $B\cup X'$. If $H$ has a matching of size $k$, then it cannot be extended to a perfect matching of $G$
(because there are not enough vertices to match vertices in $A$). So the matching number of $H$ is at most $k-1$.
By {the K\"{o}nig-Ore Formula}, 
\[
	\alpha'(H)=|B|-\max_{S\subseteq B}(|S|-|N_H(S)|)\le k-1.
\]
So there exists a subset $S\subseteq B$ such that $|S|-|N_H(S)|\ge |B|-(k-1)$. Since $|B|\ge k$, $|S|\ge 1$.
Moreover, $|N_H(S)|\le |S|+(k-1)-|B|$.
Therefore,
\[
	|N_G(S)|\le |A|+|N_H(S)|\le |A|+|S|+(k-1)-|B|=|S|,
\] 
where the last equality follows from \autoref{equ:1}.
Since $|A|\le |X|-k$, it follows that $|B|=|N(A)|=|A|+(k-1)\le |X|-1$.
Hence, $S\subseteq B$ violates the condition for $G$ to be 1-extendible.
\end{proof}

\section{Concluding Remarks}

The fact that our proof of \autoref{prop:subgraphs of $k$-extendible graphs} does not use \autoref{lem:k-extendible graphs are (k+1)-connected} makes our new proof of \autoref{lem:k-extendible graphs are (k+1)-connected} self-contained.
To the best of our knowledge, our first proof of \autoref{thm:characterization of $k$-extendible bipartite graphs}
is new and self-contained. 
The second proof, in essence, is the graph counterpart of the proof given in \cite{BP71} stated in matrix language.
That proof used the Frobenius-K\"{o}nig Theorem which is the matrix counterpart of the the K\"{o}nig-Ore Formula.
However, we feel that it may be convenient for graph theorists to have a graph-theoretical proof. 
So we include our second proof here as well.





\end{document}